\documentclass[reqno,12pt]{amsart}
\usepackage{a4wide}
\usepackage{mathrsfs}

\usepackage[T1]{fontenc}
\usepackage[utf8]{inputenc}

\newtheorem{theo}{Theorem}

\newtheorem*{prop}{Proposition}
\newtheorem{lem}{Lemma}
\theoremstyle{remark}

\email{marc.prevost (a) univ-littoral.fr}

\begin{document}
	\title[]{{Expansion of generalized Stieltjes constants in terms of derivatives of  Hurwitz  zeta-functions.%
	} }
	\author[]{M. Pr\'evost}
	\address{M. Pr\'evost,\newline
Univ. Littoral Côte d’Opale, UR 2597-LMPA-Laboratoire de Mathématiques Pures et Appliquées Joseph Liouville, F-62100, Calais, France
\newline
CNRS, FR2037, Calais, France}
	 
	\date{\today }
	
	\begin{abstract}
	Generalized	Stieltjes constants $\gamma_n(a)$ are the coefficients in the Laurent series
		for the Hurwitz-zeta function $\zeta(s,a)$ at the pole $s=1$.
		Many authors proved formulas for these constants. In this paper, using a recurrence     between $(\zeta(s+j,a))_{j}$ and  proved by the author, we prove a general result which  contains some of these formulas as particular cases.
		
	\end{abstract}

	\maketitle
	
	\section{Introduction}

	For $a\in \mathbb{C}$, $\Re(a)>0,$ the Hurwitz zeta function is defined as 
	\begin{equation*}
		\zeta (\sigma,a)=\sum_{k=0}^{\infty }\frac{1}{(k+a)^{\sigma}},	\quad\Re (\sigma)>1.
	\end{equation*}
 It can be analytically continued to $\sigma\in \mathbb{C}\setminus
	\{1\}$, with a pole at $\sigma=1$. The generalized Stieltjes constants $\gamma
	_n(a)$ occur as coefficients in the Laurent series expansion of $\zeta (\sigma,a)$
	at the pole $\sigma=1$: 
	\begin{equation}\label{Laurentzeta}
		\zeta (\sigma,a)=\frac{1}{\sigma-1}+\sum_{\ell=0}^{\infty }\frac{(-1)^{\ell}}{\ell!}
		\gamma _{\ell}(a)(\sigma-1)^{\ell}.
	\end{equation}

For $a=1$, $\zeta(\sigma,a)$ is the usual zeta-function $\zeta(\sigma)$.

	There are numerous representations for them, for instance \cite{Berndt-RMJ,wilton},
	\begin{equation*}
		\gamma _{\ell}(a)=\lim_{m\rightarrow \infty }\left\{ \sum_{k=0}^{m}\frac{\ln
			^{\ell}(k+a)}{k+a}-\frac{\ln ^{\ell+1}(m+a)}{\ell+1}\right\} ,\quad
		\ell=0,1,2,\ldots ,a\neq 0,-1,-2,\ldots
	\end{equation*}
	If $a=1$, the generalized Stieltjes constant $\gamma _{\ell}(a)$ is the
	usual Stieltjes constant $\gamma _{\ell}$.

	The series for Hurwitz zeta function converges absolutely for $\Re(\sigma)=\alpha>1$ and the convergence is uniform in the half plane $\alpha\geq \alpha_0>1$. So $\zeta(\sigma,a)$
 is analytic in the half-plane 	$\Re(\sigma)=\alpha>1$. 
	
Numerical approximations of $\gamma_\ell(a)$ have been recently given in \cite{adell2017,blagouchine2019,prevostrivoal2023}.
In the last reference, Padé approximation of the remainder term of a series giving Stieltjes constant provides a new approximation of this constant.

The goal of  the current paper is to write new series for $\gamma_\ell(a)$ whose terms are the derivatives of the zeta function    at integers. These results will generalize the formula for Euler constant proved in \cite{prevost-elsner}.

	In \cite{reczeta}, using the Padé approximation $[n,m]_{e^{x}} =\frac{\sum_{j=0}^n \frac{(-n)_j}{(-m-n)_j}\frac{x^j}{j!}}{\sum_{j=0}^m \frac{(-m)_j}{(-m-n)_j}\frac{(-x)^j}{j!}}$ of the exponential function in the  following integral representation of $\zeta(\sigma,a)$,
	
	\begin{equation}\label{intrepreszeta}
		\zeta(\sigma,a)=\frac{1}{\Gamma(\sigma)}\int_0^\infty \frac{x^{\sigma-1}e^{-ax}}{1-e^{-x}}dx  
		\end{equation}
	we proved the following result: 
	
 For $n ,m \in \mathbb{N} ,\sigma \in \mathbb{C} \setminus D_{n ,m}$ where $D_{n ,m}:=\{ -n -m ,1 -\max (n ,m) ,\cdots  , -1 ,0 ,1\}\text{,}$ for $(p ,a) \in \mathbb{C}^{2}$, let us define the quantity $A_{(n ,m)}^{(p)} (\sigma ,a)$ by

\begin{equation}
	\label{eq1}
	A_{(n ,m)}^{(p)} (\sigma ,a)\,:= \sum _{j =0}^{m}\frac{( -m)_{j}}{( -n -m)_{j}} p^{j} \frac{(\sigma)_{j}}{j !} \zeta  (\sigma\/ +\/j ,a\/ +\/p) -\sum _{j =0}^{n}\dfrac{( -n)_{j}}{( -n -m)_{j}} ( -p)^{j} \frac{(\sigma)_{j}}{j !} \zeta  (\sigma +j ,a).
\end{equation}
%	For sake of simplicity, we use the notation
%\[\alpha_j:=\dfrac{( -m)_{j}}{( -n -m)_{j}}\frac{p^j}{j!},
%\beta_j:=\dfrac{( -n)_{j}}{( -n -m)_{j}}\frac{(-p)^j}{j!}.\]

\begin{theo}	\label{teo1}
	 If $\Re  (a) >0$ and $\Re  (a +p) >0,$ then \[A_{(n ,m)}^{(p)} (\sigma ,a) =( -1)^{n +1} p^{m +n +1} \frac{(\sigma)_{m +n +1}}{(n +m) !}\int_{0}^{1}x^{m} (1 -x)^{n} \zeta  (\sigma +m +n +1 ,a +px) d x\]
	where $(\sigma)_j:=\sigma(\sigma+1)\cdots(\sigma+j-1)=\frac{\Gamma(\sigma+j)}{\Gamma(\sigma)}$ is the Pochhammer symbol.
\end{theo}
Remark 1. The conditions $\Re(a) > 0$ and$\Re(a+p) > 0$ are not restrictive since $\zeta(\sigma,a) = \zeta(\sigma,a+k)+\sum_{i=0}^{k-1}(i+a)^{-\sigma}
 > 0  $ and $\Re(a+k)>0$  for suitable integer k.
 
The convergence of  $A_{(n,m)}^{(p)}$ when $n$ or $m$ tends to infinity  has been  proved by the following Lemma.

\begin{lem}
	\label{lemma1}
	If $\left| p\right|  \leq \left| a\right|  ,\left| p  \right| \leq \left| a+p  \right|, \Re(a)>0$ and $\Re(a+p)>0,$ then $\forall  \sigma\in \mathbb{C}\setminus
\mathbb{Z}^-, \sigma\neq 1$
\[\left|\int_0^1 x^m (1-x)^n \zeta(\sigma+m+n+1,a+px)dx\right|\leq C \frac{1}{\left|a+p \right|^m }\frac{1}{\left|a \right|^n }\]
where $C$ is some constant independent of $m$ and $n$.

\end{lem}
So, under conditions on $a$ and $p$, the previous lemma shows that the quantity
$A_{(n,m)}^{(p)}(\sigma,a)$ tends to 0 when $m$ or $n$ tends to infinity.

In this paper, we   use Theorem \ref{teo1}  to construct formulas for Stieltjes constants.
To do that, we will use the following expression for Stieltjes constant 
\[\gamma_\ell(a)=(-1)^l\left( \frac{d}{d\sigma}\right)^\ell \left.\left( \zeta(\sigma,a)-\frac{1	}{\sigma-1} \right) \right| _{\sigma=1} .  \]

\section{Results}

In this section, we prove a general formula for Stieltjes constant which depends on two parameters: $n$ is the degree of the numerator and $m$ the degree of the denominator of the Padé approximant used to approximate the function $e^{-x}$ in  (\ref{intrepreszeta}).

We denote by $\zeta ^{(r ,0)} (j ,a)$ the $r$-{th} derivative of the function $\zeta(\sigma ,a)$ with respect to the variable $\sigma$, computed at $\sigma=j$.

\begin{theo}\label{teo2}
	Suppose that $\Re(a)>0$ and $\left| a\right| \geq 1$. Then
	
\begin{align}(-1)^\ell \gamma _{\ell} (a)&= \sum _{j =2}^{\max (m ,n)}\frac{( -m)_{j} -( -1)^{j} ( -n)_{j}}{( -n -m)_{j}} \frac{( -1)^{j }}{j !} \sum _{k =0}^{\ell}s (j ,k +1) \frac{\ell !}{(\ell -k) !} ( -1)^{k} \zeta ^{(\ell -k ,0)} (j ,a)+\nonumber\\ &\sum _{k =0}^{\ell}\frac{\ell !}{(\ell -k) !} \ln ^{\ell -k} a \sum _{j =1}^{m}\frac{( -m)_{j}}{( -n -m)_{j}} \frac{( -1)^{j  +\ell +1} a^{ -j}}{j !} s (j ,k +1)\nonumber\\\label{gammal(a)}& + \frac{\ln ^{\ell +1} a}{\ell +1}(-1)^{\ell+1}+{ R}_{\ell ,m ,n ,a },
\end{align} 
\end{theo}
\noindent where $s(j,k)$ are the Stirling numbers of the first kind defined by $(x)_n=\sum_{k=0}^n  (-1)^{n-k} s(n,k) x^k$ and the remainder term  is
\begin{equation}{R}_{\ell ,m ,n ,a} =\frac{( -1)^{m +1}}{(m +n) !} \sum _{k =0}^{\ell}s (m +n +1 ,k +1) \frac{\ell !}{(\ell -k) !} ( -1)^{k} \int _{0}^{1}x^{m} (1 -x)^{n} \zeta ^{(\ell -k ,0)} (m +n +1 ,a +x)dx.
\end{equation} 

The Stirling numbers satisfies also some known properties used in this paper:

\begin{align}
	\sum_{n=0}^\infty \frac{x^n}{n!}{s(n,k)}&=\frac{(\ln(1+x))^k}{k!},\;\;\;\left| x\right| <1,\label{log-stirling}\\
	s(n,1)&=(-1)^{n-1}(n-1)!,\label{sn1}\\
	s(n,2)&=(-1)^n(n-1)!H_{n-1}\label{sn2},
\end{align}
where $H_{j}=1+\frac{1}{2}+\cdots+\frac{1}{j}=\int _{0}^{1}\frac{1 -t^{j }}{1 -t} d t$
is the Harmonic number.

 Particular case.
 $\ell=0$
 \begin{align} \label{gamma(a)} \gamma(a)&= \nonumber-\sum _{j =2}^{\max (m ,n)}\frac{( -m)_{j} -( -1)^{j} ( -n)_{j}}{( -n -m)_{j}} \frac{1}{j}    \zeta (j ,a) +  \sum _{j =1}^{m}\frac{( -m)_{j}}{( -n -m)_{j}} \frac{ a^{ -j}}{j }- \\& - \ln a+{ R}_{0 ,m ,n ,a },
 \end{align} 
where
$$ {R}_{0 ,m ,n ,a} =( -1)^{n}   \int \nolimits_{0}^{1}x^{m} (1 -x)^{n} \zeta (m +n +1 ,a+x)dx.$$

For $a=1$,  formula (\ref{gamma(a)}) becomes 
\begin{equation}
\gamma=-\sum _{j =2}^{\max (m ,n)}\frac{( -m)_{j} -( -1)^{j} ( -n)_{j}}{( -n -m)_{j}} \frac{1}{j}    \zeta (j ) +  \sum _{j =1}^{m}\frac{( -m)_{j}}{( -n -m)_{j}} \frac{ 1}{j }+\varepsilon_{n,m},
\end{equation}
where $\varepsilon_{n,m}=( -1)^{n}   \int_{0}^{1}x^{m} (1 -x)^{n} \zeta (m +n +1 ,1+x)dx.$

The previous formula  has been proved in 
\cite[p. 512]{prevost-elsner}.

\section{proof of Theorem \ref{teo2}}

First, for $p=1$, 
we divide   the relation (\ref{eq1})  by $s$,     we replace $\sigma$ by $\sigma -1$ and we compute the  $\ell$-th derivative at $\sigma=1$.

	\begin{align}
	\frac{ A_{(n ,m)}^{(1)} (\sigma-1
		,a)}{\sigma-1}&=\sum _{j
		=0}^{m}\frac{(-m)_{j}}{(-n -m)_{j}} \frac{(\sigma)_{j-1}}{j !}
	\zeta  (\sigma-1+j ,a+1) -\sum _{j
		=0}^{n}\frac{(-n)_{j}(-1)^j}{(-n -m)_{j}}
	\frac{(\sigma)_{j-1}}{j !} \zeta (\sigma -1 +j
	,a)\label{eq3}
	\\& =(-1)^{n +1} \frac{(\sigma )_{m +n}}{(n +m) !} \int \nolimits_{0}^{1}x^{m}
	(1 -x)^{n} \zeta  (\sigma +m +n ,a +x) d x.\label{rel7}
\end{align}

  In the relation (\ref{eq3}), the term for $j=0$ is $\frac{1	}{\sigma-1}\zeta(\sigma-1,a+1)-\frac{1	}{\sigma-1}\zeta(\sigma-1,a)=\frac{-1	}{\sigma-1} a^{-\sigma+1}$.
  
  For $j=1$, it is $\frac{-m}{-n-m} \zeta(\sigma,a+1)-\frac{n}{-n-m} \zeta(\sigma,a)=\zeta(\sigma,a)-\frac{m}{n+m}a^{-\sigma}$.

Thus 
\begin{align*}
\left( \frac{d}{d\sigma}\right)^\ell \left.  \frac{A_{(n ,m)}^{(1)} (\sigma -1 ,a)}{\sigma-1}  \right| _{\sigma=1}&=\left.\left( \frac{d}{d\sigma}\right)^\ell\left(\zeta  (\sigma ,a) -\frac{a^{ -\sigma+1}}{\sigma -1} -\frac{m}{n +m}a^{ -\sigma}\right) \right| _{\sigma=1} \\&+\sum _{j =2}^{m}\frac{( -m)_{j}}{( -n -m)_{j}} \frac{1}{j !}\left.   \left( \frac{d}{d\sigma}\right)  ^\ell(\sigma)_{j -1}\zeta  (\sigma -1 +j ,a +1) \right| _{\sigma=1} \\&-\sum _{j =2}^{n}\frac{( -n)_{j}}{( -n -m)_{j}} \frac{( -1)^{j}}{j !}\left.\left( \frac{d}{d\sigma}\right)^\ell(\sigma)_{j -1}\zeta  (\sigma -1 +j ,a) \right| _{\sigma=1}.
\end{align*}

 Using the relation $\zeta(\sigma,a+1)=\zeta(\sigma,a)-a^{-\sigma}$, we find
 \begin{align*}
 	\left( \frac{d}{d\sigma}\right)^\ell\left. \frac{A_{(n ,m)}^{(1)} (\sigma -1 ,a)}{\sigma-1}\right| _{\sigma=1}&=\left( \frac{d}{d\sigma}\right)^\ell\left. \left(\zeta  (\sigma ,a) -\frac{a^{ -\sigma+1}}{\sigma -1}\right)\right| _{\sigma=1} \\&+\sum _{j =2}^{m}\frac{( -m)_{j}-(-n)_j(-1)^j}{( -n -m)_{j}} \frac{1}{j !}\left. \left( \frac{d}{d\sigma}\right)^\ell(\sigma)_{j -1}\zeta  (\sigma -1 +j ,a )\right| _{\sigma=1} \\&
- \sum _{j =1}^{m}\frac{( -m)_{j}}{( -n -m)_{j}} \frac{1}{j !}\left. \left( \frac{d}{d\sigma}\right)^\ell(\sigma)_{j -1} a^{-\sigma-j+1} \right| _{\sigma=1}	.
 \end{align*}

Now, we use the formula  
\cite[Lemma 1]{coffeyRocky2014}

\[\left. \left( \frac{d}{d\sigma}\right)^\ell(\sigma)_{j -1}\right| _{\sigma=1} =s (j ,\ell +1) \ell !( -1)^{j -1 +\ell} ,j \geq 1.\]

Then, with the product rule, we  obtain
\[\left. \left( \frac{d}{d\sigma}\right)^\ell(\sigma)_{j -1}\zeta  (\sigma -1 +j ,a)\right| _{\sigma=1} =\sum _{k =0}^{\ell}s (j ,k +1) \frac{\ell !}{(\ell -k) !} ( -1)^{j-1+k} \zeta ^{(\ell -k )} (j ,a) ,j \geq 2,\]

and 

\[\left. \left( \frac{d}{d\sigma}\right)^\ell(\sigma)_{j -1}a^{ -\sigma -j +1}\right| _{\sigma=1} =\sum _{k =0}^{\ell}\frac{\ell !}{(\ell -k) !} (\ln a) ^{\ell -k} ( -1)^{j +\ell +1}a^{ -j}s (j ,k +1) ,j \geq 1,\]

 \begin{align*}
 \left. \left( \frac{d}{d\sigma}\right)^\ell\left( \zeta  (\sigma ,a) -\frac{a^{ -\sigma+1}}{\sigma -1}\right)\right| _{\sigma=1}&=\left( \frac{d}{d\sigma}\right)^\ell\left. \left( \zeta  (\sigma ,a) -\frac{1}{\sigma-1}+\frac{1-a^{ -\sigma+1}}{\sigma -1}\right)\right| _{\sigma=1}\\ &=( -1)^{\ell} \gamma _{\ell} (a) +( -1)^{\ell}\frac{\ln^{\ell +1} a }{\ell +1}.
 \end{align*}

It arises

\begin{align*}
	( -1)^{\ell} \gamma _{\ell} (a)& = -\frac{\ln ^{\ell +1} a}{\ell +1} ( -1)^{\ell}+\sum _{k =0}^{\ell}\frac{\ell !}{(\ell -k) !} \ln ^{\ell -k} a \sum _{j =1}^{m}\frac{( -m)_{j}}{( -n -m)_{j}} \frac{( -1)^{j +\ell +1} a^{ -j}}{j !} s (j ,k +1)\\&+ \sum _{j =2}^{\max (m ,n)}\frac{( -m)_{j} -( -1)^{j} ( -n)_{j}}{( -n -m)_{j}} \frac{( -1)^{j +1}}{j !} \sum _{k =0}^{\ell}s(j ,k +1) \frac{\ell !}{(\ell -k) !} ( -1)^{k} \zeta ^{(\ell -k ,0)} (j ,a) 
 +R_{\ell ,m ,n ,a}
\end{align*}  
%(voir coffey 1.12 Series representations for the Stieltjes constants prop 4 de rocky mountain )
where

\begin{equation}\label{eq9}
	R_{\ell ,m ,n ,a} =\frac{( -1)^{m +1}}{(m +n) !} \sum _{k =0}^{\ell}s (m +n +1 ,k +1) \frac{\ell !}{(\ell -k) !} ( -1)^{k} \int \nolimits_{0}^{1}x^{m} (1 -x)^{n} \zeta ^{(\ell -k ,0)} (m +n +1 ,a +x).
\end{equation} 
	The derivation of $\frac{ A_{(n ,m)}^{(1)} (\sigma-1 ,a)}{\sigma-1}$ with respect to $\sigma$ is justified by Lemma \ref{lem2} since the function $\phi(x) $ is integrable.
	
In the following, we slightly modify       Lemma \ref{lemma1} for $p=1$.
\begin{lem}
	\label{lemma3}
	Suppose that $\Re(a)>0$ and $\left| a\right| \geq 1$. Then $\forall k\in[0,\ell], \forall m,n $ such that $m+n\geq l+1$,
	\[\left|\int_0^1 x^m (1-x)^n \zeta(m+n+1+k-\ell,a+x)dx\right|\leq K \frac{1}{\left|a+1 \right|^{m} }\frac{1}{\left|a \right|^n },\]
	where $K$ is some constant independent of $m$ and $n$.
\end{lem}

\begin{proof}
 
 \[	\left|\int_0^1 x^m (1-x)^n \zeta(m+n+1+k-\ell,a+x)dx\right|\leq \int_0^1  x^m(1-x)  ^n  \left| \zeta(m+n+1+k-\ell,a+x)\right|dx\]
 \[\leq\int_0^1   \left| \frac{x}{a+x} \right| ^m \left| \frac{1-x}{a+x} \right| ^n \left|a+x\right|^{l-k-1} \left| \zeta(m+n+1+k-\ell,a+x)\right|\left|a+x\right|^{m+n+1+k-l}dx.\]
 
  In \cite{reczeta}, we proved that for $\Re(z)\geq 1, \left|  \zeta(z,a+x)(a+x)^z\right| $ is bounded by some constant $C$ independent of $z$ and that
  $\max_{0\leq x\leq 1}\left| \frac{x}{a+x}\right| =\frac{1}{\left|a+1\right|  }$ and  $\max_{0\leq x\leq 1}\left| \frac{1-x}{a+x}\right| =\frac{1}{\left|a\right|  }$.
   This proves the Lemma, where $K=C\int_{0}^{1} \left|a+x \right|^{l-k-1} dx.$

\end{proof}

In order to prove the convergence to 0 of the remainder term $R_{l,m,n,a}$ of our principal formula (\ref{gammal(a)}), we need a preliminary technical result.

\begin{lem}\label{lem2}
	Suppose that  $r \in \mathbb N,$ $a \in \mathbb C, \Re(a)>0$ and $\left| a\right| \geq 1$, $j> r+1$.
	
	Then $\forall x\geq 0, \left| a+x\right|\geq 1 $, and
	\begin{align}
		\left| 	   \zeta^{(r,0)}(j,a+x)\right| &\leq \sum_{k\geq 0}\frac{\left( (\ln\left|k+a+x \right|)^2+\pi^2/4\right) ^{r/2} }{\left|k+a+x \right|^{j}}=:\phi(x)\label{eq7}\\
		\phi(x)&\leq2^{r}\zeta(j-r,\Re(a)+x)\label{eq8}
	\end{align}
	\begin{proof}
		\begin{align*}
			\left| 	   \zeta^{(r,0)}(j,a+x)\right|&= \left| \sum_{k\geq 0}(-1)^r\frac{\left( \ln(k+a+x \right) ^{r} }{(k+a+x )^{j}}\right| 
			\leq \sum_{k\geq 0}\frac{\left|  s(k+a+x )\right|  ^{r} }{\left| k+a+x\right|^{j}}
			\\&=\sum_{k\geq 0}\frac{\left|  \ln\left|k+a+x\right| +i Arg(k+a+x) \right|  ^{r} }{\left| k+a+x\right|^{j}}\\&=
			\sum_{k\geq 0}\frac{   (\ln\left|k+a+x\right|)^2 + (Arg(k+a+x))^2 )  ^{r/2} }{\left| k+a+x\right|^{j}}
			\\&\leq
			\sum_{k\geq 0}\frac{   (\ln\left|k+a+x\right|)^2 + (\pi/2)^2 )  ^{r/2} }{\left| k+a+x\right|^{j}} \leq
			\sum_{k\geq 0}\frac{ ( \left|k+a+x\right|^2 +  3 \left|k+a+x\right|^2 )  ^{r/2} }{\left| k+a+x\right|^{j}} \\&\leq
			2^r \sum_{k\geq 0}\frac{   \left|k+a+x\right|^{r}} {\left| k+a+x\right|^{j}}=
			2^r \sum_{k\geq 0}\frac{   1} {\left| k+a+x\right|^{j-r}}\\&\leq
			2^r \sum_{k\geq 0}\frac{   1} { (k+\Re(a)+x)^{j-r}}=	2^r  \zeta(j-r,\Re(a)+x)
		\end{align*}
		
	\end{proof}
	
\end{lem}

\begin{theo}\label{teo3}
	
	If $\Re(a)\geq 1$  then
\[\forall n \in \mathbb{N} ,\lim_{m\rightarrow \infty }	R_{\ell ,m ,n ,a}=0,\]

	\[\forall m \in \mathbb{N},\lim_{n\rightarrow \infty }	R_{\ell ,m ,n ,a}=0.\]

\end{theo}

\begin{proof}
  From (\ref{eq9}), we can write 
\[	\left| R_{\ell ,m ,n ,a}\right|  \leq\frac{1}{(m +n) !} \sum _{k =0}^{\ell}\left| s (m +n +1 ,k +1)\right|  \frac{\ell !}{(\ell -k) !}  \int \nolimits_{0}^{1}x^{m} (1 -x)^{n}\left|  \zeta ^{(\ell -k ,0)} (m +n +1 ,a+x)\right| dx \]
\end{proof}

Adell in \cite{adell2022}, proved some explicit upper bounds for the Stirling numbers of the first kind. We will use the following formula:
 \[  for \,\, l=1,\cdots,j-1, \;\;\;\; \left| s(j+1 ,\ell +1)\right|  \leq  \frac{j!}{\ell !}(\ln j )^{\ell}  \left( 1 +\frac{\ell}{\ln j} \right). \]

\bigskip \bigskip 
\begin{align}
	\left| R_{\ell ,m ,n ,a}\right|  &\leq \sum _{k =0}^{\ell} (\ln(m+n))^k\left(1+\frac{k}{\ln(m+n)}\right)  \frac{1}{(\ell -k) !}  \int _{0}^{1}x^{m} (1 -x)^{n}\left|  \zeta ^{(\ell -k ,0)} (m +n +1 ,a+x)\right| dx  \nonumber\\ 
	&\leq \left(1+\frac{\ell}{\ln(m+n)}\right)(\ln(m+n))^\ell
	\sum _{k =0}^{\ell}2^{\ell-k}\int_0^1 x^m(1-x)^n \zeta(m+n+1-\ell+k,\Re(a)+x)\nonumber\\(\rm{Lemma\; } \ref{lem2})\nonumber\\
	&\leq C \left(1+\frac{\ell}{\ln(m+n)}\right)(\ln(m+n))^\ell 2^{\ell+1}
  \frac{1}{(\Re(a)+1 )^{m} }\frac{1}{\Re(a) ^n }.
  \label{eq12}
\end{align}

The second inequality is valid since $m$ and/or $n$ tends to infinity and thus the parameter $m+n+1+k-\ell$ is greater that 1.

Thus, if $\Re(a)> 1$, the limit of the remainder $R_{\ell,mn,a}$ is $0$ when $m$ or $n$  tend to infinity and we have
\[\lim_{m\rightarrow \infty }\left| R_{\ell,m,n,a}\right|^{1/m} \leq \frac{1}{\Re(a)+1},
\]

\[\lim_{n\rightarrow \infty }\left| R_{\ell,m,n,a}\right|^{1/n} \leq \frac{1}{\Re(a)}.
\]
If $\Re(a)=1,$ then $\lim_{m\rightarrow \infty }\left| R_{\ell,m,n,a}\right|^{1/m} \leq 1/2 $.

Now, if $\Re(a)=1$ and $n$ tends to infinity, we have to bound the integral term of  right hand side of (\ref{eq12}) as following.
\[\int_0^1x^m  (1-x)^n \zeta(m+n+1+k-\ell,1+x)dx\leq
\int_0^1(1-x)^n \zeta(m+n+1+k-\ell,1+x)dx\]
\[\leq \int_0^1\frac{(1-x)^n}{(1+x)^{m+n+1+k-\ell}} (1+x)^{m+n+1+k-\ell} \zeta(m+n+1+k-l,1+x)dx\]
\[\leq \int_0^1\frac{(1-x)^n}{(1+x)^{n}} (1+x)^{m+n+1+k-\ell} \zeta(m+n+1+k-\ell,1+x)dx\]
\[\leq C\int_0^1\frac{(1-x)^n}{(1+x)^{n}}  dx\leq  \frac{C}{2n} . \] (see Lemma \ref{lemma3}).

We can conclude that, when $\Re(a)=1$ the remainder term tends to $0$ when $n$ tends to infinity.

\vskip 2cm

{\bf Remark}

To accelerate the computation, we can use the following  relation (for $p$ integer):

\[\gamma _{\ell} (a) =\gamma _{\ell} (a +p) +\sum _{k =0}^{p -1}\frac{\ln ^{\ell} (a +k)}{a +k}.\]

It leads to
\begin{align*}
( -1)^{\ell} \gamma _{\ell} (a)& =( -1)^{\ell} \gamma _{\ell} (a +p) +( -1)^{\ell} \sum _{k =0}^{p -1}\frac{\ln ^{\ell} (a +k)}{a +k}\\
& =( -1)^{\ell} \sum _{k =0}^{p -1}\frac{\ln ^{\ell} (a +k)}{a +k}+\\& \sum _{j =2}^{\max (m ,n)}\frac{( -m)_{j} -( -1)^{j} ( -n)_{j}}{( -n -m)_{j}} \frac{( -1)^{j +1}}{j !} \sum _{k =0}^{\ell}S (j ,k +1) \frac{\ell !}{(\ell -k) !} ( -1)^{k} \zeta ^{(\ell -k ,0)} (j ,a +p) +\\
&\sum _{k =0}^{\ell}\frac{\ell !}{(\ell -k) !} \ln ^{\ell -k} (a +p) \sum _{j =1}^{m}\frac{( -m)_{j}}{( -n -m)_{j}} \frac{( -1)^{j +\ell +1} (a +p)^{ -j}}{j !} S (j ,k +1) +R_{\ell ,m ,n ,a +p},
\end{align*}

where the remainder  terms $R_{\ell ,m ,n ,a +p}$ converges to 0 as $\displaystyle \frac{1}{\Re(a+p+1)^m \Re(a+p)^n}$
\vskip 2cm

 \section{Particular cases}
 
 In this section, we consider all the possible values for the four  parameters $ m, n,a$.

\vskip 2cm
\subsection{$m \in \mathbb{N},n=0, \Re(a)\geq 1$}

After simplification, we get

\begin{align}(-1)^\ell \gamma _{\ell} (a)&= \nonumber\sum _{j =2}^{m}  \frac{( -1)^{j }}{j !} \sum _{k =0}^{\ell}s (j ,k +1) \frac{\ell !}{(\ell -k) !} ( -1)^{k} \zeta ^{(\ell -k ,0)} (j ,a)\\ &+\sum _{k =0}^{\ell}\frac{\ell !}{(\ell -k) !} \ln ^{\ell -k} a \sum _{j =1}^{m}\frac{( -1)^{j  +\ell +1} a^{ -j}}{j !} s (j ,k +1)\\ \label{gammaL}& +\nonumber \frac{\ln ^{\ell +1} a}{\ell +1}(-1)^{\ell+1}+{ R}_{\ell ,m ,0 ,a }.
\end{align} 

If $m$ tends to infinity, it arises  (see Theorem \ref{teo3})

\begin{equation}
	\begin{split}
		(-1)^\ell \gamma _{\ell} (a) = {} & \sum _{j =2}^{\infty}  
		\frac{( -1)^{j }}{j !} \sum _{k
			=0}^{\ell}s (j ,k +1)
		\frac{\ell !}{(\ell -k) !}(-1)^{k} \zeta
		^{(\ell -k ,0)} (j ,a)    \\ \label{gamma2(a)}
		& +\sum
		_{k=0}^{\ell}\frac{\ell !}{(\ell -k) !}
		\ln ^{\ell -k} a \sum
		_{j=1}^{\infty}\frac{( -1)^{j+\ell +1}
			a^{ -j}}{j !} s (j ,k +1) \\
		& + \frac{\ln
			^{\ell +1} a}{\ell +1}(-1)^{\ell+1}.
	\end{split}
\end{equation}

After simplification (using relation (\ref*{log-stirling})) leads to

\begin{align}(-1)^\ell \gamma _{\ell} (a)&= \frac{\ln ^{\ell +1} (a-1)}{\ell +1}(-1)^{\ell+1}+ \sum _{j =2}^{\infty}  \frac{( -1)^{j }}{j !} \sum _{k =0}^{\ell}s (j ,k +1) \frac{\ell !}{(\ell -k) !} ( -1)^{k} \zeta ^{(\ell -k ,0)} (j ,a).\nonumber
\end{align} 
\vskip 2cm

\subsection{$n \in \mathbb{N},m=0, \Re(a)\geq 1$}

After simplification, we get

\begin{align}(-1)^\ell \gamma _{\ell} (a)&=& \nonumber\sum _{j =2}^{n}  \frac{1}{j !} \sum _{k =0}^{\ell}s (j ,k +1) \frac{\ell !}{(\ell -k) !} ( -1)^{k} \zeta ^{(\ell -k ,0)} (j ,a)+ \frac{\ln ^{\ell +1} a}{\ell +1}(-1)^{\ell+1}+{ R}_{\ell ,0 ,n ,a }.
\end{align} 

If $n$ tends to infinity, it arises (see Theorem \ref{teo3})
\begin{eqnarray}\label{gamma-n-infty}(-1)^\ell \gamma _{\ell} (a)&=& \nonumber\sum _{j =2}^{\infty}  \frac{1}{j !} \sum _{k =0}^{\ell}s (j ,k +1) \frac{\ell !}{(\ell -k) !} ( -1)^{k} \zeta ^{(\ell -k ,0)} (j ,a)+ \frac{\ln ^{\ell +1} a}{\ell +1}(-1)^{\ell+1}.
\end{eqnarray} 

\vskip 2cm

\section{Case $\Re(a)> 0,\left|a \right|\geq 1, m=\lambda n, m,n \rightarrow \infty$}
Now, we consider the case   $m=\lambda n, $ when $ n $ tends to infinity. 

Of course, the general expression of Theorem \ref{teo2} is true and we can derive the following series

\begin{align*}
 ( -1)^{\ell} \gamma _{\ell} (a)& =  
 -\frac{\ln ^{\ell +1} a}{\ell +1} ( -1)^{\ell} +\frac{\lambda }{\lambda  +1} \frac{\ln ^{\ell} a}{a} ( -1)^{\ell} \\&-\sum _{j =2}^{\infty }\frac{\left (\lambda \right )^{j} -( -1)^{j}}{\left (\lambda  +1\right )^{j}} \frac{( -1)^{j +1}}{j !} \sum _{k =0}^{\ell}s (j ,k +1) \frac{\ell !}{(\ell -k) !} ( -1)^{k} \zeta ^{(\ell -k ,0)} (j ,a)\\& +\sum _{k =0}^{\ell}( -1)^{\ell +1} \frac{\ell !}{(\ell -k) !} \ln ^{\ell -k} (a) \sum _{j =2}^{\infty }\left( \frac{ -\lambda /a}{\lambda  +1}\right) ^{j} \frac{1}{j !} s (j ,k +1).
 \end{align*}
 
Using the relation (\ref{log-stirling}) and for  $\left|  \frac{ \lambda /a}{\lambda  +1}\right|   <1  $, it can be simplified as 

\begin{equation}\label{eqlambda}
( -1)^{\ell} \gamma _{\ell} (a)=( -1)^{\ell +1} \frac{\ln ^{\ell +1} \left (a -\frac{\lambda }{\lambda  +1}\right )}{\ell +1} -{\displaystyle\sum _{k =0}^{\ell}}\frac{\ell ! ( -1)^{k}}{(\ell -k) !} \left [{\displaystyle\sum _{j =2}^{\infty }}\frac{\left (\lambda \right )^{j} -( -1)^{j}}{\left (\lambda  +1\right )^{j}} \frac{( -1)^{j +1}}{j !} s (j ,k +1) \zeta ^{(\ell -k ,0)} (j ,a)\right ] .
\end{equation}

One may consider (\ref{eqlambda})
as
an identity between two holomorphic functions in the variable $\lambda$. 
Since this formula is true for $\lambda \in \mathbb{Q}$ such that  $\left|  \frac{ \lambda }{\lambda  +1}\right|   <\left| a\right|   $, it is
obviously true for every $\lambda \in \mathbb{C}, \left|  \frac{ \lambda }{\lambda  +1}\right|   <\left| a\right|  . $

Now, suppose that  $a$ is real greater than 1  and we choose $\lambda  =i$.

Equating real and imaginary parts of (\ref*{eqlambda}) we get
\begin{align*}
	( -1)^{\ell} \gamma _{\ell} (a)&=( -1)^{\ell +1} \Re \left(\frac{\ln ^{\ell +1} \left (a -1/2 -i/2\right )}{\ell +1}\right)\\
 &-{\sum _{k =0}^{\ell}}\frac{\ell ! ( -1)^{k}}{(\ell -k) !} \left [\begin{array}{c}{\sum _{q =1}^{\infty }}( -1)^{q} 2^{ -2 q} \frac{1}{(4 q +1) !} s (4 q +1 ,k +1) \zeta ^{(\ell -k ,0)} (4 q +1 ,a) \\
 -{\sum _{q =0}^{\infty }}( -1)^{q} 2^{ -2 q -1} \frac{1}{(4 q +3) !} s (4 q +3 ,k +1) \zeta ^{(\ell -k ,0)} (4 q +3 ,a)\end{array}\right ]
\end{align*}
and the following  identity between Hurwitz-$\zeta$ functions:
\begin{align}
	\sum_{k=0}^\ell \frac{l!	}{(l-k)!}(-1)^k \sum_{j \equiv 2(4)}^{\infty}(-1/4)^{\frac{j-2}{4}}\frac{(-1)^{j-1}}{j!}s(j,k+1)\zeta ^{(\ell -k ,0)}(j,a)=\Im\left(
	(-1)^{\ell+1}\frac{\ln(a-\frac{i+1}{2})}{\ell+1}^{\ell+1}
	\right),
\end{align}
or more simply
\begin{align}
	\sum_{k=0}^\ell \frac{l!	}{(l-k)!}(-1)^k \sum_{r=0}^{\infty}\frac{(-1/4)^{r}}{(4r+2)!}s(4r+2,k+1)\zeta ^{(\ell -k ,0)}(4r+2,a)=\Im\left(
	(-1)^{\ell}\frac{\ln(a-\frac{i+1}{2})}{\ell+1}^{\ell+1}
	\right).
\end{align}

\section{Particular case: $\ell=1$}
In this section, we consider the first Stieltjes constant $\gamma_1$.
We will show a new formula of this constant in terms of the first derivatives of the $\zeta$-functions.

\begin{prop}
\[\gamma_1=\sum_{j=1}^{\infty}\frac{\zeta'(2j+1)}{2j+1}.\]
\end{prop}

\begin{proof}

	We use the formula (\ref{gamma-n-infty}) with $a=1$, $l=1$ and $\lambda=0$.
\begin{equation}  \gamma _{1} = \sum _{j =2}^{\infty}  \frac{1}{j !}(s(j,1)\zeta'(j)-s(j,2)\zeta(j))
\end{equation} 
  The formula (\ref{gamma-n-infty}) with $a=2$, $l=1$ and $\lambda\rightarrow\infty$.
\begin{align}  \gamma_1(2)=\gamma _{1} &= \sum _{j =2}^{\infty}  \frac{1}{j !}(-1)^{j-1}(s(j,1)\zeta'(j,2)-s(j,2)\zeta(j,2))\\&= \sum _{j =2}^{\infty}  \frac{1}{j !}(-1)^{j-1}(s(j,1)\zeta'(j)-s(j,2)\zeta(j,2))
\end{align} 
Using the expression of the first Stirling numbers (\ref{sn1}, \ref*{sn2}), the two expression of $\gamma_1$ become

\begin{equation}  \gamma _{1} = \sum _{j =2}^{\infty}  \frac{(-1)^{j-1}}{j }(\zeta'(j)+
	H_{j-1}\zeta(j))
\end{equation} 

\begin{equation} \gamma _{1} = \sum _{j =2}^{\infty}  \frac{1}{j }(\zeta'(j)+H_{j-1}\zeta(j,2)) 
\end{equation} 
In the sequel, we will  show that 
\[\sum _{j =2}^{\infty}\frac{H_{j-1}}{j}  \left(  \zeta(j,2)+(-1)^{j-1}
\zeta(j)\right) =0\] 

and the proposition will be  proved.

\[\sum _{j =2}^{\infty}\frac{H_{j-1}}{j}  \left(  \zeta(j,2)+(-1)^{j-1}
\zeta(j)\right)  =\int _{0}^{1}\sum _{j =2}^{\infty}\frac{1}{j}  \left(  \zeta(j,2)+(-1)^{j-1}
\zeta(j)\right) \frac{1 -t^{j -1}}{1 -t} d t.\]

The permutation  is valid since the convergence of the series is uniform on $[0,1]$.

 Using the formula (https://dlmf.nist.gov/5.7.E3)

\begin{equation*}
\ln\Gamma\left(1+z\right)=-\ln\left(1+z\right)+z(1-\gamma)+\sum_{k=2}^{\infty}
(-1)^{k}\zeta\left(k,2\right)\frac{z^{k}}{k}, \;\;\;\left| z\right|<2,
\end{equation*}

it is easy to prove that
\begin{align*}
\sum _{j \geq 2}\frac{1}{j} (\zeta  (j ,2) +( -1)^{j-1} \zeta  (j)) \frac{1 -t^{j -1}}{1 -t}& = \sum _{j \geq 2}\frac{1}{j} \zeta  (j ,2) \frac{1 -t^{j -1}}{1 -t}+\sum _{j \geq 2}\frac{ ( -1)^{j-1}}{j} \zeta  (j) \frac{1 -t^{j -1}}{1 -t}\\&=-\frac{\ln \Gamma(2-t)}{t(1-t)}+\frac{\ln \Gamma(1+t)}{t(1-t)}=:\phi(t).
\end{align*}
This function $\phi$ is symetric with respect to the abscissa $t=1/2$, so its integral between $0$ and $1$ is zero and the proposition is proved. 
\end{proof}

\end{document}